\documentclass[11pt]{article}
\usepackage{fullpage}
\usepackage{amsmath,amsthm,amssymb,ifthen}

\newboolean{@full}
\setboolean{@full}{true}

\newcommand{\iffull}[2]{\ifthenelse{\boolean{@full}}{ #1}{ #2}}

\newtheorem{theorem}{Theorem}[section]
\newtheorem{corollary}[theorem]{Corollary}
\newtheorem{lemma}[theorem]{Lemma}

\theoremstyle{definition}
\newtheorem{definition}[theorem]{Definition}

\newdimen\pIR
\newcommand\StevesR{{\rm I\kern\pIR R}}

\newcommand\symExpec{\operatorname{\mathbb{E}}\displaylimits}

\def\expec#1{\symExpec_{#1} }

\newcommand{\E}{\mathbb{E}}

\def\setof#1{\left\{#1  \right\}}

\newcommand{\R}{\mathbb{R}}

\newcommand{\mydet}[1]{\det\left(#1\right)}

\usepackage{tikz,amstext}

\newcommand*\myrect[1]{%
  \begin{tikzpicture}
    \node[draw,rectangle,inner sep=0pt] {#1};
  \end{tikzpicture}}

\newcommand{\sqsum}{\mathbin{\text{\myrect{\textnormal{+}}}}}
\newcommand{\recsum}{\mathbin{\textnormal{\myrect{++}}}}
\newcommand{\subsquare}{\mathbb{S}}

\newcommand{\cauchy}[2]{\mathcal{G}_{#1} \left(#2 \right)}
\newcommand{\invcauchy}[2]{\mathcal{K}_{#1} \left(#2 \right)}

\newcommand{\charp}[2]{\chi_{#2} \left(#1 \right)}

\newcommand{\onevec}{\mathbf{1}}

\newcommand{\rstt}{(R_\theta\oplus I_{d-2})}

\newcommand{\dil}[1]{\begin{bmatrix} 0 & #1 \\ #1^T & 0\end{bmatrix}}
\newcommand{\block}[2]{\begin{bmatrix} #1 & 0 \\ 0 & #2\end{bmatrix}}

\title{
Interlacing Families IV: Bipartite Ramanujan Graphs of All Sizes
\thanks{
This research was partially supported by NSF grant CCF-1111257,
  an NSF Mathematical Sciences Postdoctoral Research Fellowship, Grant No. DMS-0902962,
  a Simons Investigator Award to Daniel Spielman, and a MacArthur Fellowship.
}
\iffull{}{(Extended Abstract)}}

\author{
Adam W. Marcus\\
Crisply\\
Yale University\\
\and
Daniel A. Spielman \\ 
Yale University
\and 
Nikhil Srivastava\\
UC Berkeley
}

\date{\today}
\begin{document}
\maketitle

\begin{abstract}
We prove that there exist bipartite Ramanujan graphs of every degree and every
  number of vertices.
The proof is based on analyzing the expected characteristic polynomial of a
  union of random perfect matchings, and involves three ingredients: 
  (1) a formula for the expected characteristic polynomial of the sum of a regular
  graph with a random permutation of another regular graph,
  (2) a proof that this expected polynomial is real rooted and that
  the family of polynomials considered in this sum is an interlacing family, 
  and
  (3) strong bounds on the roots of the expected characteristic polynomial
  of a union of random perfect matchings,
  established using the framework of finite free convolutions introduced recently in
  \cite{mssfinite}.
\end{abstract}


\newpage

\section{Introduction}
Ramanujan graphs are undirected regular graphs whose nontrivial adjacency matrix
  eigenvalues are asymptotically as small as possible; in other words, they are the optimal spectral
  expander graphs.
In this paper, we prove the existence of bipartite Ramanujan graphs of every
  degree and every size.
We do this by showing that a random $m-$regular bipartite graph, obtained as a
  union of $m$ random perfect matchings across a bipartition of an even number
  of vertices, is Ramanujan with nonzero probability.
Specifically, we prove that the expected characteristic polynomial of
  such a random graph has roots concentrated in the appropriate range, and
  use the method of interlacing families introduced in \cite{IF1} to deduce that 
  there must be an actual graph whose eigenvalues are no worse than the roots of this polynomial.
Infinite families of bipartite Ramanujan graphs were shown in that paper to exist for every
  degree $m\ge 3$ , but it was not known whether they exist for every number of vertices.

The main conceptual and technical contributions of this work and the companion
  paper \cite{mssfinite} are the following.
First, we identify a new class of real-rooted expected characteristic
  polynomials related to random graphs, and develop new tools for establishing
  their interlacing properties and analyzing the locations of their roots.
These methods are different from those used to study the mixed characteristic
  polynomials of \cite{IF2}, and the bounds we obtain are strictly stronger than those
  produced by the original ``barrier method'' argument introduced in \cite{bss} 
  (which is off by a factor of two in this setting).
Notably, the expected characteristic polynomials we consider are computable
  in polynomial time, unlike most other known expected characteristic
  polynomials.
Second, in contrast to previous work, we derive the Ramanujan bound from
  completely generic considerations involving random orthogonal matrices, in 
  particular making no use of results from algebraic graph theory 
  or number theory.

\subsection{Summary of Results}
Recall that the adjacency matrix $A$ of an \textbf{$\mathbf{m}-$regular graph on $\mathbf{d}$ vertices}\footnote{%
{In order 
  to be consistent with our companion paper \cite{mssfinite}, we will,
  unconventionally, use $m$ to denote the degree of a graph and $d$ to denote its
  number of vertices.}}  has largest eigenvalue
  $\lambda_1(A)=m$, and smallest eigenvalue $\lambda_d(A)=-m$ when the graph is
  bipartite.
Following Friedman \cite{friedman}, we will refer to these as the {\em trivial}
  eigenvalues of $A$, and we will call a graph {\em Ramanujan} if all of its
  non-trivial eigenvalues have absolute value at most $2\sqrt{m-1}$.
Such graphs are asymptotically best possible in the sense that a theorem of Alon and Boppana
  \cite{Nilli} tells us that for every $\epsilon>0$, every infinite sequence
  of $m-$regular graphs must contain a graph with a non-trivial eigenvalue of
  absolute value at least $2\sqrt{m-1}-\epsilon$.

Our main theorem is that a union of $m$ random perfect matchings across a
  bipartition of $2d$ vertices is Ramanujan with nonzero probability.
\begin{theorem}\label{thm:bipartite} Let $P_1,\ldots,P_m$ be independent
	uniformly random $d\times d$
	permutation matrices, $m\ge 3$. Then, with nonzero probability the nontrivial eigenvalues of 
$$ A=\sum_{i=1}^m \begin{bmatrix} 0 & P_i \\ P_i^T & 0\end{bmatrix}$$
	are all less than $2\sqrt{m-1}$ in absolute value.
\end{theorem}

We also prove the following non-bipartite version of this theorem, regarding a
  union of $m$ random perfect matchings on $d$ vertices (not bipartite), with
  $d$ even.
\begin{theorem}\label{thm:nonbipartite}
Let $P_1,\ldots,P_m$ be independent uniformly random $d\times d$ permutation
  matrices, $d$ even, $m\ge 3$. Let $M$ be the adjacency matrix of any fixed
  perfect matching on $d$ vertices. Then with nonzero probability:
$$ \lambda_2\left(\sum_{i=1}^m P_iMP_i^T\right)<2\sqrt{m-1}.$$
\end{theorem}

Since we only prove nonzero bounds on the probabilities, the nonbipartite
  theorem is a logical consequence of the bipartite one.
We  describe it
  here because its proof is substantially easier and contains most of the main
  ideas.
Note that Theorem \ref{thm:nonbipartite} does not produce Ramanujan graphs
  because it does not guarantee any control of the least eigenvalue $\lambda_d$. 

We remark that as they are unions of independent matchings, the graphs we produce
  may have multiple edges between two vertices.
Thus, they are strictly speaking multigraphs, and do not subsume the previous
  results if one insists on simple graphs.
However, it seems that it should be more difficult to construct Ramanujan graphs with multiedges
  than without.
Like \cite{IF1}, this paper establishes existence but does not give a polynomial
  time construction of Ramanujan graphs.

\subsection{Related Work} 
Infinite families of Ramanujan graphs were 
  first shown to exist for $m=p+1$, $p$ a prime, in the seminal work of Margulis and
  Lubotzky-Phillips-Sarnak \cite{margulis,lps}.
The graphs they produce are Cayley graphs and can be constructed very
  efficiently, and their analysis relies on deep results from number theory, which
  is responsible for the ``Ramanujan'' nomenclature.
Friedman \cite{friedman} showed that a random $m-$regular graph is almost 
  Ramanujan: specifically, that a union of $m$ perfect matchings has non-trivial
  eigenvalues bounded by $2\sqrt{m-1}+\epsilon$ with high probability, for every
  $\epsilon > 0$.
More recently, in \cite{IF1}, we proved the existence of infinite families of $m-$regular bipartite
  Ramanujan graphs for every $m\ge 3$ by proving (part of) a conjecture of Bilu
  and Linial \cite{bilulinial} regarding the existence of good $2-$lifts of
  regular graphs.
Prior to the present paper, it was unknown if there exist Ramanujan graphs of every number of vertices.
We refer the reader to \cite{hlw} and \cite{IF1} for a more detailed discussion
  of Ramanujan graphs and $2-$lifts.

\subsection{Outline of the Paper}
The proofs of both of our theorems follow the same strategy and consist of three
  steps.
In each step we present the simpler non-bipartite case first, and then indicate
  the modifications required for the bipartite case.
  
First, we show that the expected characteristic polynomials of the 
  random graphs we are interested in are real rooted and come from interlacing families (reviewed in
  Section \ref{sec:prelim}), which
  reduces our existence theorems to analyzing the roots of these polynomials.
This is achieved in
  Section \ref{sec:inter} by decomposing the random permutations used to
  generate these expected polynomials into swaps acting on two vertices at a time, and showing
  that such swaps correspond to linear transformations which preserve
  real-rootedness properties of certain multivariate polynomials.
Theorem \ref{thm:swapreal} implies that if $A$ and $B$ are symmetric matrices,
  then the expected characteristic polynomial of $A + P B P^{T}$ is real rooted
  for a random permutation matrix $P$.
We remark that this argument is completely elementary and self-contained, and unlike
  \cite{IF1,IF2} does not appeal to any results from the theory of real stable
  or hyperbolic polynomials.
In the process, we introduce a class of ``determinant-like'' polynomials which
  may be of independent interest.

Next, in Section \ref{sec:quad}
  we derive a closed-form formula for the expected characteristic polynomial
  of a sum of randomly permuted regular graphs.
We begin by proving that the expected characteristic
  polynomials over random permutations can be replaced by expected characteristic
  polynomials over random orthogonormal matrices.
This may be seen as a quadrature or derandomization statement, which says that
  these characteristic polynomials are not able to distinguish between the set of
  permutation matrices and the set of orthogonal matrices; essentially this happens
  because determinants are multilinear, which causes certain restrictions of
  them to have very low degree Fourier coefficients.
This component of the proof may also be of independent interest.

Finally, we appeal to machinery developed in our companion paper
  \cite{mssfinite}, which studies the structure of expected 
  characteristic polynomials over random orthogonal matrices.
 In particular, such polynomials may be expressed crisply in terms of a simple
  (and explicitly computable) convolution operation on characteristic polynomials,
  which we call the finite free additive convolution. 
In this framework, the
  characteristic polynomial of  a union of $m$ random matchings is simply 
  the $m-$wise convolution of the characteristic polynomial of a single matching.
By applying strong bounds on the roots of these convolutions derived in
  \cite{mssfinite}, we obtain the desired Ramanujan bound of $2\sqrt{m-1}$.
The requisite material regarding free convolutions is introduced in Sections
\ref{sec:freeprelim} and \ref{sec:cauchyintro}.
	
{These three ingredients are combined in Section \ref{sec:bounds} to complete the
  proofs of Theorems \ref{thm:bipartite} and \ref{thm:nonbipartite}.}

\section{Preliminaries}
\subsection{Interlacing Families}\label{sec:prelim}
We recall the following theorem from \cite{IF2}, stated here in the slightly
  different language of product distributions.
\begin{theorem}[Interlacing Families]\label{thm:if} Suppose $\{f_\omega(x)\}_{\omega\in \{0,1\}^m}$ is a family of real-rooted polynomials
	of the same degree $n$ with positive leading coefficient, such that
	$$ E_\mu(x):=\E_{\omega\sim\mu} f_\omega(x)$$
	is real-rooted for every product distribution
	$\mu=\mu_1\otimes\ldots\otimes \mu_m$ on $\Omega=\{0,1\}^m$. Then for
	every $k=1,\ldots,n$ and every such $\mu$, there is some
	$\omega_0\in\Omega$ such that
	$$\lambda_k(f_{\omega_0})\le\lambda_k(E_\mu),$$
where $\lambda_k$ denotes the $k$th largest root of a real-rooted polynomial.
\end{theorem}

For real rooted polynomials $f$ and $g$, we write
  $g \longrightarrow f$ if the roots of $f$ and $g$ interlace and the largest
  root of $f$ is at least as big as the largest root of $g$.
We will use the following elementary facts about interlacing and
  real-rootedness, which may be found in \cite{fisk}.
\begin{lemma}\label{lem:fisk} 
	If $g$ has degree one less than $f$ and both are real-rooted, then 
	\begin{enumerate}
		\item	$g\longrightarrow f$ if and only if $f+\alpha g$ is real-rooted for all $\alpha\in\R$
		\item   $g\longrightarrow f$ implies that $f\longrightarrow f-g$.
	\end{enumerate}

	If $f_{1}$ and $f_{2}$ are monic and real-rooted of the same degree, then they have a common interlacing
	if and only if $f_{1}+\alpha f_{2}$ is real-rooted for all $\alpha\ge 0$.
\end{lemma}
\subsection{Finite Free Convolutions of Polynomials}\label{sec:freeprelim}
To analyze the expected characteristic polynomials of the random graphs we
  consider, we will need the notion of a {\em finite free convolution} of two
  polynomials, developed in our companion paper \cite{mssfinite}.
We denote the characteristic polynomial of a matrix by:
  $$\charp{A}{x}:=\det(xI-A).$$

\begin{definition}[Symmetric Additive Convolution] 
Let $p(x)=\charp{A}{x}$ and $q(x)=\charp{B}{x}$ be two real-rooted polynomials, for some symmetric
   $d\times d$ matrices $A$ and $B$. The {\em symmetric additive convolution} of $p$ and $q$ is defined as:
\[
  p (x) \sqsum_{d} q (x) = \expec{Q} \charp{A + Q B Q^{T}}{x},
\]
where the expectation is taken over random orthogonal matrices $Q$ sampled
according to the Haar measure on $O(d)$, the group of $d$-dimensional orthonormal matrices.
\end{definition}
Note that this is a well-defined operation on polynomials because the distribution of the
  eigenvalues of $A+QBQ^T$ depends only on the eigenvalues of $A$ and the eigenvalues of $B$, 
  which are the roots of $p$ and $q$.

\begin{definition}[Asymmetric Additive Convolution] 
Let $p(x)=\charp{AA^T}{x}$ and $q(x)=\charp{BB^T}{x}$ be two real-rooted
  polynomials with nonnegative roots, for some arbitrary (not necessarily
  symmetric) $d\times d$ matrices $A$ and $B$. The {\em asymmetric additive
  convolution} of $p$ and $q$ is defined as
\[
	p (x) \recsum_{d} q (x) = \expec{Q,R} \charp{(A + Q B R^{T})(A+QBR^T)^T}{x},
\]
where $Q$ and $R$ are independent random orthogonal matrices sampled uniformly from 
  $O(d)$.
\end{definition}
When dealing with a possibly asymmetric $d\times d$ matrix $M$, we will frequently consider the
  {\em dilation}
$$\dil{M},$$
which is by construction a symmetric $2d\times 2d$ matrix.
We will refer to a matrix of this type as a {\em bipartite} matrix.
It is easy to see that its eigenvalues are symmetric about $0$ and are equal to
  $\pm\lambda_1(MM^T)^{1/2},\ldots,\pm\lambda_d(MM^T)^{1/2}$, i.e., in absolute value to the 
  singular values of $M$.
This correspondence also gives the useful identity
\begin{equation}\label{eqn:dil} \subsquare\charp{MM^T}{x} =
  \charp{\dil{M}}{x},\end{equation}
where the operator $\subsquare$ is defined by $$(\subsquare p)(x):=p(x^2).$$
With this notation in hand, we can alternately express the asymmetric additive convolution as
\begin{equation}\label{eqn:asymdil}
	\subsquare (p (x) \recsum_{d} q (x)) = \expec{Q,R}
	\charp{\dil{A}+\block{Q}{R}\dil{B}\block{Q}{R}^T}{{x}}.
\end{equation}

Explicit, polynomial time computable formulas for the additive convolutions in terms of the coefficients of $p$ and
  $q$ may be found in Theorems 1.1 and 1.3 of \cite{mssfinite}.
For this work, we only require the following important consequences of these
  formulas, also established in \cite{mssfinite}.
We will occasionally drop the subscripts in $\sqsum_d$ and $\recsum_d$ when it is clear
  from the context.
\begin{lemma}[Properties of $\sqsum$ and $\recsum$]

	\begin{enumerate}
		\item If $p(x)$ and $q(x)$ are real-rooted then $p(x)\sqsum_d q(x)$ is
			also real-rooted.
		\item If $p(x)$ and $q(x)$ are real-rooted with all roots
			nonnegative, then $p(x)\recsum_d q(x)$ is also real-rooted
			with all roots nonnegative.

		\item The operations $\sqsum_d$ and $\recsum_d$ are  bilinear
			\iffull{(in the coefficients of the polynomials on which
			they operate) and associative.}{and associative.}
\end{enumerate}
\end{lemma}
\iffull{\begin{proof} 
	(1) and (2) are Theorems 1.2 and 1.4 of \cite{mssfinite}, and 
	bilinearity follows immediately from Theorems 1.1
	and 1.3 of \cite{mssfinite}. To see associativity, let
	$p(x)=\charp{A}{x}, q(x)=\charp{B}{x}$ and $r(x)=\charp{C}{x}$, and observe
	that
	\begin{align*}
		(p(x)\sqsum q(x))\sqsum r(x) 
		&= \left(\expec{Q}\expec{R} \charp{QAQ^T+RBR^T}{x}\right) \sqsum \charp{C}{x}
		\\&= \expec{Q}\expec{R} \left(\charp{QAQ+RBR^T}{x}\sqsum \charp{C}{x}\right)\quad\textrm{by bilinearity}
		\\&= \expec{Q} \expec{R} \expec{W} \charp{QAQ+RBR^T+WCW^T}{x},
	\end{align*}
	for random orthogonal matrices $Q,R,W$. The same argument shows that
	this is also equal to $p(x)\sqsum (q(x)\sqsum r(x))$.

	An analogous argument using the formula \eqref{eqn:asymdil} shows that
	$\recsum$ is also associative.
\end{proof}
}{}
A consequence of the above lemma is that for $m$ matrices $A_1,\ldots,A_m$,
  identities such as
  \begin{equation}\label{eqn:msumid} \expec{Q_1,\ldots,Q_m}\charp{\sum_{i=1}^m Q_iA_iQ_i^T}{x} =
  \charp{A_1}{x}\sqsum\charp{A_2}{x}\sqsum\ldots\sqsum\charp{A_m}{x}\end{equation}
make sense.
\subsection{Cauchy Transforms}\label{sec:cauchyintro}
The device that we use to analyze the roots of finite free convolutions of
  polynomials is the Cauchy Transform. 
This is the same (up to normalization) as the Stieltjes Transform and
 the ``Barrier Function'' of \cite{bss,IF1,IF2}.
  \begin{definition}[Cauchy Transform] The {\em Cauchy Transform} of a
	  polynomial $p(x)$ with roots $\lambda_1,\ldots,\lambda_d$ is defined
	  to be the function
\[
	\cauchy{p}{x} = \frac{1}{d} \sum_{i=1}^{d} \frac{1}{x-\lambda_{i}} =
	\frac{1}{d}\frac{p'(x)}{p(x)}.
\]
	We define the {\em inverse Cauchy Transform} of $p$ to be
\[
  \invcauchy{p}{w} = \max \setof{x : \cauchy{p}{x} = w}.
\]
\end{definition}
Note that the Cauchy transform has poles at the roots of $p$, and
when all the roots $\lambda_i$ of $p$ are real, $\cauchy{p}{x}$ is monotone
  decreasing for $x$ greater than the largest root.
Thus, $\invcauchy{p}{w}$ is the unique value of $x$ that is larger than all 
  the $\lambda_{i}$ for which $\cauchy{p}{x} = w$.
In particular, it is an upper bound on the largest root of $p$, and approaches the
  largest root as $w\rightarrow\infty$.

Our bounds on the expected characteristic polynomials of random graphs are a consequence of the following two
  theorems, which are proved in \cite{mssfinite}. 
  \begin{theorem}[Theorem 1.7 of \cite{mssfinite}]\label{thm:sqsumTrans} For real-rooted degree $d$ polynomials $p$
	and $q$ and $w > 0$,
\[
  \invcauchy{p \sqsum_{d} q}{w} \leq  \invcauchy{p}{w} + \invcauchy{q}{w} - 1/w.
\]
\end{theorem}
The above theorem is a strengthening of the univariate barrier function argument
  for characteristic polynomials introduced in \cite{bss}.
This may be seen by taking $q(x)=\charp{B}{x}=x^{d-1}(x-d)$, which corresponds to a
  rank one matrix $B=vv^T$ with trace equal to $d$. It is easy to check that in
  this case $p(x)\sqsum q(x)=p(x)-p'(x)$.

\iffull{We remark that Theorem \ref{thm:sqsumTrans} is inspired by an {\em equality} regarding inverse Cauchy transforms of
  limiting spectral distributions of certain random matrix models arising in
  Free Probability theory; we refer the interested reader to \cite{mssfinite}
  for a more detailed discussion.
}{}
To analyze the case of bipartite random graphs, we will need the corresponding
  inequality for the asymmetric convolution.
  \begin{theorem}[Theorem 1.8 of \cite{mssfinite}]\label{thm:recsumTrans}
For degree $d$ polynomials $p$ and $q$ having only nonnegative real roots,
\[
	\invcauchy{\subsquare( p \recsum_{d} q)}{w} \leq
  \invcauchy{\subsquare p}{w} + \invcauchy{\subsquare q}{w}-1/w.
\]
\end{theorem}

\section{Interlacing for Permutations}\label{sec:inter}
In this section, we show that the expected characteristic polynomials
  obtained by averaging over certain random permutation matrices form interlacing
  families. 
The class of random permutations which have this property are those that are products of independent
  random swaps, which we now formally define.
\begin{definition}[Random Swap] A {\em random swap} is a matrix-valued random
	variable which is equal to a transposition of two (fixed) indices $s,t$
	with probability $\alpha$ and equal to the identity with probability
	$(1-\alpha)$, for some $\alpha\in [0,1]$.
\end{definition}
\begin{definition}[Realizability by Swaps] A matrix-valued random variable $P$
supported on permutation matrices is {\em realizable by swaps} if there are
random swaps $S_1,\ldots,S_N$ such that the distribution of $P$ is the same as
the distribution of the product $S_NS_{N-1}\ldots S_2S_1$. 
\end{definition}
For example, we show in Lemma \ref{lem:realizable} 
  below that a uniformly random permutation matrix is realizable by swaps.

The main result of this section is that expected characteristic polynomials over
  products of random swaps are always real-rooted. 
These polynomials play a role
  analogous to that of mixed characteristic polynomials in \cite{IF1,IF2}.
\begin{theorem} \label{thm:swapreal} Let $A_1,\ldots,A_m$ be symmetric $d\times d$ matrices and let
$\{S_{ij}\}_{i\le m, j\le N}$ be independent (not necessarily identical) random swaps. Then the expected
characteristic polynomial
\begin{equation}\label{eqn:swappol} \E \det\left(tI-\sum_{i=1}^m 
  \left(\prod_{j=N}^1 S_{ij} \right) A_i \left(\prod_{j=1}^N S_{ij}^T \right)\right)\end{equation}
is real-rooted.
\end{theorem}
An immediate consequence of Theorems \ref{thm:swapreal} and 
  \ref{thm:if}, applied to the family of polynomials indexed by all possible
  values of the swaps $S_{ij}$, is the following existence result.
\begin{theorem}[Interlacing Families for Permutations]\label{thm:perminterlace} Suppose $A_1,\ldots,A_m$ are symmetric
$d\times d$ matrices, and $P_1,\ldots,P_m$ are independent random
permutations realizable by swaps. Then, for every $k\le d$:
$$ \lambda_k\left(\sum_{i=1}^m P_i A_i P_i^T\right)\le
\lambda_k\left(\E\charp{\sum_{i=1}^m P_i A_i P_i^T}{x}\right),$$
with nonzero probability.
\end{theorem}
Theorem \ref{thm:perminterlace} is useful because the uniform
  distribution on permutations and its bipartite version, which we use to
  generate our random graphs, are realizable by swaps.
\begin{lemma} \label{lem:realizable} Let $P$ and $S$ be uniformly random
	$d\times d$ permutation matrices. Both $P$ and $P\oplus S$ are
	realizable by swaps, where $P \oplus S = \begin{pmatrix}
P & 0 \\
0 & S
\end{pmatrix}$ is the direct sum of $P$ and $S$.
\end{lemma}
\iffull{
\begin{proof} We will establish the claim for $P$ first. We proceed inductively.
	Let $M_2$ be a random swap which swaps $e_1$ and $e_2$ with
	probability $1/2$, and for $k>2$ let 
	$$M_k=M_{k-1}S_{1k}M_{k-1},$$
	where $S_{1k}$ swaps $e_1$ and $e_k$ with probability $1/k$. 
	
	Let $v=(1,2,3,\ldots,d)^T$. By induction, assume that the first $k-1$
	coordinates of $M_{k-1}v$ are in uniformly random order; in particular,
	that $(M_{k-1}v)(1)$ is a random element of $\{1,\ldots, k-1\}.$ 
	This means that:
	\begin{itemize}
		\item With probability $1/k$: $(M_{k-1}S_{1k}M_{k-1}v)(k)=k$ and
			the remaining indices contain a random permutation of
			$\{1,\ldots,k-1\}$.
		\item With probability $1-1/k$: $(M_{k-1}S_{1k}M_{k-1}v)(k)$ is
			a uniformly random element $j\in \{1,\ldots,k-1\}$ and
			the remaining indices contain a random permutation of
			$\{1,\ldots,k\}\setminus \{j\}.$
	\end{itemize}
	Thus, $M_k$ is uniformly random on $\{1,\ldots,k\}$, and by induction
	$M_d=P$.

	For $P\oplus S$, we use the above argument to realize $P\oplus I$ and
	$I\oplus S$ separately and then multiply them.
\end{proof}
}{See the full version for a proof.}

The rest of this section is devoted to proving Theorem \ref{thm:swapreal}. 
This is achieved by showing that the polynomials in \eqref{eqn:swappol} are
  univariate restrictions of certain nice multivariate polynomials.
\iffull{The relevant notion is the following.}{}

\begin{definition}[Determinant-like Polynomials] A homogeneous polynomial $P(X_1,\ldots,X_m)$ of degree $d$ in
	the entries of $m$ symmetric $d\times d$ matrices $X_1,\ldots,X_m$ is
	called {\em determinant-like} if it has the following two properties.
\begin{enumerate}
\item []\textbf{Hyperbolicity.} The univariate restrictions
	$$ q(t)=P(tI-A_1,\ldots, tI-A_m)$$
are real-rooted for all symmetric $A_1,\ldots, A_m$. 

This condition is known as {\em hyperbolicity} of the polynomial
  $P(X_1,\ldots,X_m)$ with respect to the point $(I,I,\ldots, I)$.
We do not discuss the notion of hyperbolicity further, since the self-contained 
  definition above suffices for this paper. We point the
  interested reader to \cite{pemantle} for a detailed discussion of
  the theory.

\item [] \textbf{Rank-1 Linearity.} For every vector $v$, index $i\le m$, and real number $s$, we have
$$ P(X_1,X_2,\ldots, X_i+svv^T,\ldots, X_m) = P(X_1,\ldots,X_m)+ sD_{i,vv^T}P(X_1,\ldots,X_m)$$
where
$$ D_{i,vv^T}P(X_1,\ldots,X_m)= \left(\frac{\partial}{\partial s}
P(X_1,\ldots,X_i+svv^T,\ldots, X_m)\right)\big|_{s=0}$$
is the directional derivative of $P$ in direction $(0,\ldots, vv^T,\ldots,0)$,
where $vv^T$ appears in the $i$th position. 
Note that $D_{i,vv^T}P(X_1,\ldots,X_m)$ is homogeneous of degree $d-1$.
\end{enumerate}
\end{definition}

An important example of a determinant-like polynomial is the determinant of a sum of
matrices:
$$P(X_1,\ldots,X_m) = \det(X_1+\ldots+X_m).$$
Hyperbolicity follows from the fact that
$$P(tI-A_1,\ldots,tI-A_m)=\det(mtI-A_1-\ldots-A_m)$$
is the characteristic polynomial of a symmetric matrix. 
Rank-1 linearity can be seen to follow from the invariance of the determinant under change of basis
  and its linearity with respect to matrix entries.
Alternatively, one can prove it by using the
 matrix
  determinant lemma, which tells us 
$$ \det(X_1+svv^T+\ldots +X_m)=\det(X_1+\ldots+X_m)+s\langle
vv^T,\det(X_1+\ldots+X_m)(X_1+\ldots+X_m)^{-1}\rangle.$$

The crux of the proof of Theorem \ref{thm:swapreal} lies in the fact that random swaps define linear operators
  which preserve the property of being determinant-like.
  \begin{lemma}[Random swaps preserve determinant-likeness]\label{lem:rswap} If $P(X_1,\ldots,X_m)$ is determinant-like, then
	for any $i\le m$ and random swap $S$, the polynomial
	$$ \E_{S} P(X_1,\ldots, SX_iS^T, \ldots, X_m)$$
	is determinant-like.
\end{lemma}

\iffull{
Before proving this lemma, we record some preliminary facts about
  determinant-like polynomials.
}{
This is proved using the following results, whose proofs appear in the full version.
}
  \begin{lemma}[Rank-1 updates interlace]\label{lem:detlikeinter} Suppose $P(X_1,\ldots, X_m)$ is determinant-like. Then
for every vector $v$ and symmetric matrices $A_1,\ldots, A_m$ we have
	$$P(tI-A_1,\ldots,tI-A_m)\longrightarrow P(tI-A_1,\ldots,
	tI-A_i-vv^T,\ldots tI-A_m),$$
	where $\longrightarrow$ denotes interlacing, pointing to the polynomial
	with the largest root.
\end{lemma}
\iffull{\begin{proof} Assume without loss of generality that $i=1$.
By rank-1 linearity,
$$ P(tI-A_1-svv^T,\ldots,tI-A_m) =
P(tI-A_1,\ldots,tI-A_m)-sD_{vv^T}P(tI-A_1,\ldots,tI-A_m).$$
By the hyperbolicity of $P$, we know that this is real rooted when viewed as a univariate polynomial in
  $t$.
Since $D_{1,vv^T}P$ is of degree one less than $P$, the first part of Lemma \ref{lem:fisk}  implies that
$$ D_{1,vv^T}P(tI-A_1,\ldots,tI-A_m)\longrightarrow P(tI-A_1,\ldots,tI-A_m),$$
which in turn by the second part of Lemma \ref{lem:fisk} gives
\begin{align*}
	P(tI-A_1,\ldots,tI-A_m) & \longrightarrow  P(tI-A_1,\ldots,tI-A_m) -
D_{1,vv^T}P(A_1-tI,\ldots,A_m-tI)\\
&=P(tI-A_1-vv^T,\ldots,tI-A_m),\end{align*}
as desired.
\end{proof}
}{}
\begin{lemma}[Permutations preserve rank-1 linearity] \label{lem:rank1properties} (1) If $\Pi$ is a permutation matrix and $P(X_1,\ldots,X_m)$ is rank-1 linear then $P(\Pi X_1
\Pi^T,X_2,\ldots,X_m)$ is also rank-1 linear. (2) If $P$ and $Q$ are rank-1
linear then so is $P+Q$.\end{lemma}
\iffull{\begin{proof} (1) is true because the set of rank one matrices is invariant
under conjugation by permutations.
(2) holds because $D_{i,vv^T}$ is a linear operator.
\end{proof}
}{}
\iffull{We will also need the following elementary observation, which says that random
  swaps correspond to trace zero rank two updates.
This is the structural property which causes interlacing to occur.}
{The structural property of swaps which causes interlacing to occur is the
 following.}
\begin{lemma}\label{lem:rank2} If $\sigma$ is a transposition and $A$ is symmetric then $A-\sigma
	A\sigma^T$ has rank $2$ and trace $0$.\end{lemma}
\iffull{\begin{proof} Assume without loss of generality that $\sigma$ swaps the first two
coordinates. Then by symmetry the difference $A-\sigma A\sigma^T$ has entries
$$ \begin{bmatrix}
a_{11} - a_{22} & a_{12}-a_{21} & a_{13}-a_{23} & a_{14}-a_{24} & \ldots \\
a_{21} - a_{12} & a_{22}-a_{11} & a_{23}-a_{13} & a_{24}-a_{14} & \ldots \\
a_{31} - a_{32} & a_{32}-a_{31} & 0 & \ldots\\
a_{41} - a_{42} & a_{42}-a_{41} & 0 & \ldots \\
\ldots\\
\end{bmatrix} 
= 
\begin{bmatrix}
\alpha & \beta & v^T\\
\beta & -\alpha  & -v^T \\
v & -v & 0
\end{bmatrix}
$$
for some numbers $\alpha,\beta$ and some column vector $v$ of length $d-2$. If
$\alpha\neq\beta$ then the sum of the first two rows is equal to $(c,-c,0,\ldots,
0)$ for some $c\neq 0$, and every other row is a scalar multiple of this. On the
other hand, if $\alpha=\beta$ then the first two rows are linearly dependent, and 
all of the other rows are multiples of $(1,-1, 0, \ldots, 0).$
\end{proof}
}{}
\iffull{We can now complete the proof of Lemma \ref{lem:rswap}
\begin{proof}[Proof of Lemma \ref{lem:rswap}] Assume $P$ is determinant-like,
	and let $S$ be a random swap, equal to some transposition $\sigma$ with
	probability $\alpha$ and the identity with probability $(1-\alpha)$.
We will show that
$$ Q(X_1,\ldots, X_m) = (1-\alpha)P(X_1,\ldots,X_m) +
\alpha P(X_1,\ldots, \sigma X_i \sigma^T,\ldots, X_m),$$
is hyperbolic and rank-1 linear. It is clear that $Q(X_1,\ldots,X_m)$  is homogeneous since swaps and
convex combinations preserve homogeneity. Lemma \ref{lem:rank1properties} implies that rank-1 linearity
is also preserved, so all that remains is hyperbolicity. Assume without loss of generality that $i=1$ and consider any
univariate restriction along $(I,I,\ldots,I)$:
\begin{equation}\label{eqn:rswapQ}
 Q(tI-A_1,\ldots,tI-A_m) = 
(1-\alpha)P(tI-A_1,\ldots,tI-A_m) +
\alpha P(tI-\sigma A_1\sigma^T,\ldots,  tI-A_m).
\end{equation}
We need to show that this has all real roots. Observe that the second polynomial
may be written as 
$$ P(tI-A_1-aa^T+bb^T,\ldots,tI-A_m),$$
for some vectors $a$ and $b$ , since $\sigma A_1\sigma^T-A_1$ is rank two and
  trace zero by Lemma \ref{lem:rank2}.
Since $P$ is determinant-like, Lemma \ref{lem:detlikeinter} tells us that
$$ P(tI-A_1+bb^T,\ldots,tI-A_m)\longrightarrow
P(tI-A_1-aa^T+bb^T,\ldots,tI-A_m)$$
and
$$ P(tI-A_1+bb^T,\ldots,tI-A_m)\longrightarrow
P(tI-A_1,\ldots,tI-A_m),$$
whence the two polynomials on the right hand side of \eqref{eqn:rswapQ} 
  have a common interlacing.
Lemma \ref{lem:fisk} then implies that  their convex
  combination must be real-rooted, and the claim is proved.
\end{proof}
}{}
Applying Lemma \ref{lem:rswap} inductively yields Theorem \ref{thm:swapreal}.
\begin{proof}[Proof of Theorem \ref{thm:swapreal}]
	Applying Lemma \ref{lem:rswap} $nN$ times (once for every swap $S_{ij}$) 
starting with $P(X_1,\ldots,X_m)=\det(\sum_i X_i)$ tells us that 
$$ \E_{S_{1N}}\ldots\E_{S_{11}} \E_{S_{2N}}\ldots\E_{S_{n1}}
\det\left(\sum_{i=1}^n \left(\prod_{j=N}^1 S_{ij} \right) X_i
  \left(\prod_{j=1}^N S_{ij}^T \right) \right)$$
  is determinant-like.  Considering the restriction $X_i=(t/m)I-A_i$ finishes the proof.
\end{proof}

\section{Quadrature}\label{sec:quad}
In this section, we show that the expected characteristic polynomials 
  we are interested in are free convolutions
  of the characteristic polynomials of perfect matchings,
  after the trivial eigenvalues corresponding to the all ones vector are removed. 
This gives us explicit formulas for these polynomials, and more importantly
  (since we understand the behavior of roots under free convolutions) a way of
  bounding their roots.
We begin by showing how to do this for the symmetric case, which is more
  transparent and contains all the main ideas.
In Section \ref{sec:quadbip} we derive the result for the bipartite case as
  a corollary of the result for the symmetric case.

\subsection{Quadrature for Symmetric Matrices}

The following theorem gives an explicit formula for the expected characteristic polynomial
  of the sum of two symmetric matrices with constant row sums when the rows and columns of one of the matrices
  is randomly permuted.
This can be used to compute the expected characteristic polynomial of the Laplacian matrix
  of the sum of two graphs when one is randomly permuted.
In this paper, we use the result to compute the expected characteristic polynomial of the adjacency matrix
  when both graphs are regular.

\begin{theorem}\label{thm:symquadconv}
Suppose $A$ and $B$ are symmetric $d\times d$ matrices with $A\onevec = a\onevec$ and
$\onevec = b\onevec$. Let $\charp{A}{x}=(x-a)p(x)$ and
	$\charp{B}{x}=(x-b)q(x)$. Then,
	\begin{equation}\label{eqn:symperm}
	\E_{P}\charp{A+PBP^T}{x}=(x-(a+b))p(x)\sqsum_{d-1} q(x),\end{equation}
	where $P$ is a uniformly random permutation.
\end{theorem}
\iffull{}{This is a consequence of the following results whose proofs may be found in the full version.}

\iffull{We begin by writing \eqref{eqn:symperm} in a more concrete form.
Observe that all of the matrices $A,B,P$ have $\onevec$ as a left
  and right eigenvector, which means that there is an orthogonal change of basis $V$ (for
  concreteness,  mapping
  $\onevec$ to the standard basis vector $e_n$) that
  simultaneously block diagonalizes all of them:
\begin{equation}\label{eqn:blockdiag} VAV^{T}=\hat{A}\oplus a,\quad VBV^{T}=\hat{B}\oplus b,\quad
  	VPV^{T}=\hat{P}\oplus 1,\end{equation}
where $\hat{A}\oplus a$ denotes the direct sum
$$ \begin{bmatrix} \hat{A} & 0 \\ 0 & a\end{bmatrix}.$$
Since the determinant is invariant under change of basis, we may write
\begin{align}\label{eqn:dividehat}
	\nonumber \E_P \det(xI-A-PBP^T)
	&=\E_{P}\det(xI-VAV^{T}-(VPV^{T})(VBV^{T})(VP^TV^{T}))
	\\\nonumber &= \E_{\hat{P}}\det(xI-(\hat{A}\oplus a)-(\hat{P}\oplus 1)(\hat{B}\oplus b)(\hat{P}^T\oplus 1)) 
	\\&=(x-a-b)\E_{\hat{P}}\det(xI-\hat{A}-\hat{P}\hat{B}\hat{P}^T).
\end{align}
Notice also that $p(x)=\charp{\hat{A}}{x}$ and $q(x)=\charp{\hat{B}}{x}$, so 
$$ p(x)\sqsum q(x) = \E_Q \det(xI-\hat{A}-Q\hat{B}Q^T),$$
where $Q$ is a (Haar) random $(d-1)\times (d-1)$ orthogonal matrix.
Thus, \eqref{eqn:symperm} is equivalent to showing that 
\begin{equation}\label{eqn:hatp} \E_{\hat{P}}\det(xI-\hat{A}-\hat{P}\hat{B}\hat{P}^T) =
	\E_{{Q}}\det(xI-\hat{A}-{Q}\hat{B}{Q}^T),\end{equation}
for all $(d-1)\times (d-1)$ symmetric matrices $\hat{A},\hat{B}$. 
Note that for any permutation $P$, the orthogonal transformation $\hat{P}$ correspondingly permutes
$\hat{e_1},\ldots,\hat{e_n}$, the projections orthogonal to $\onevec$ of the standard basis vectors
$e_1,\ldots,e_d$, embedded in $\R^{d-1}$. Since these are the vertices of a
regular simplex with $d$ vertices in $\R^{d-1}$ centered at the origin, we interpret the $\hat{P}$ as
elements of the symmetry group of this simplex.
We denote this subgroup of $O(d-1)$ by $A_{d-1}$.

Since there is no longer any assumption on $\hat{A}, \hat{B}$ other than
symmetry, we may absorb the $xI$ term into $\hat{A}$ in \eqref{eqn:hatp}, and we
see that it is sufficient to establish the following.
}{}
\iffull{\begin{theorem}[Quadrature Theorem]\label{thm:mainquad} For symmetric $d\times d$ matrices $A$ and $B$,
	\begin{equation}\label{eqn:mainquad} \expec{P\in
		A_{d}}\det(A+PBP^T)=\expec{Q\in
	O(d)}\det(A+QBQ^T).\end{equation}
\end{theorem}
}{
\begin{theorem}[Quadrature Theorem]\label{thm:mainquad} Given symmetric $d\times d$ matrices $A,B$,
	\begin{equation}\label{eqn:mainquad} \expec{P\in
		A_{d}}\det(A+PBP^T)=\expec{Q\in
	O(d)}\det(A+QBQ^T),\end{equation}
where $A_{d}$ is the group of isometries of a regular simplex centered at the
  origin in $\R^{d}$.
\end{theorem}
}
\iffull{It is easy to see that the theorem will follow if we can show that the left hand
  side of \eqref{eqn:mainquad} is invariant under right multiplication of $P$ by
  orthogonal matrices.
}{}
\begin{lemma}[Invariance Implies Quadrature]\label{lem:invarquad}
Let $f$ be a function from $O (d)$ to $\R$ and
  let $H$ be a finite subgroup of $O (d)$. If
\begin{equation}\label{eqn:rotinvar}
  \expec{P \in H}{f (P)} =
  \expec{P \in H}{f (P Q_0)} .
  \end{equation}
	for all $Q_0 \in O (d)$, then
  \begin{equation}\label{eqn:quad} \expec{P\in H}f(P)=\expec{Q\in
  O(d)}f(Q),\end{equation}
where $Q$ is chosen according to Haar measure and $P$ is uniform on $H$.
\end{lemma}
\iffull{\begin{proof} 
\begin{align*}
	\expec{Q\in O(d)}{f(Q)} &= \expec{Q\in O(d)}{}\expec{P\in H}{f(PQ)}
	= \expec{P\in H}{}\expec{Q\in O(d)}{f(PQ)}
	\\&= \expec{P\in H}{}\expec{Q\in O(d)}{f(P)}
	= \expec{P\in H}{f(P)},
\end{align*}
as desired.
\end{proof}

We will prove Theorem \ref{thm:mainquad} by showing that $f(P)=\det(A+PBP^T)$
  satisfies \eqref{eqn:rotinvar}.
We will achieve this by demonstrating that $f$ is invariant under certain elementary
  orthogonal transformations acting on 3-faces of the regular simplex, which
  generate all orthogonal transformations.
Let us fix some notation to precisely describe these elementary transformations.

Given three vertices $\hat{e_i},\hat{e_j},\hat{e_k}$ of the
  regular simplex, let $A_{i,j,k}$ denote the subgroup of $A_d$
  consisting of permutations of $\hat{e_i},\hat{e_j},\hat{e_k}$ which
  leave all of the other vertices fixed.
  Let  $O_{i,j,k}$ denote the subgroup of $O(d)$ acting on the
  two dimensional linear subspace parallel to the affine subspace through the three
  vertices, and leaving the orthogonal subspace fixed.
  Note that $A_{i,j,k}$ is a subgroup of $O_{i,j,k}$, and that these groups are
  isomorphic to $A_2$ and $O(2)$, respectively.

  The heart of the proof lies in the following lemma, which implies by Lemma
  \ref{lem:invarquad} that the
  polynomials we are interested in are not able to distinguish between the uniform
  distributions on $A_{2}$ and $O(2)$.
The reason for this is that these polynomials have very low degree (at most
  two) in the entries of any orthogonal matrix $Q$ acting on a two-dimensional
  subspace, a fact which is essentially a consequence of the multilinearity of the
  determinant.
The argument below is similar to the proof of Lemma 2.7 in \cite{mssfinite}.
}{}
\begin{lemma}[Invariance for $A_2$]\label{lem:a2invar} If $A$ and $B$ are symmetric $d\times d$ matrices, then
	for every $Q_0\in O(2)$,
\begin{equation}\label{eqn:a2invar}
 \expec{P\in A_2}{\det(A+(P\oplus I_{d-2})B(P\oplus
	I_{d-2})^T)}=
	\expec{P\in A_2}{\det(A+(PQ_0\oplus I_{d-2})B(PQ_0\oplus
	I_{d-2})^T)}.
\end{equation}
\end{lemma}
\iffull{\begin{proof}
Let $SO(2)$ be the subgroup of $O(2)$ consisting of rotation matrices
	$$R_{\theta}=\begin{bmatrix} \cos\theta & \sin\theta \\ -\sin\theta &
	\cos\theta\end{bmatrix},$$
and let $Z_3$ be the subgroup of $A_2$ consisting of the three rotations
$R_\tau, \tau\in T:=\{0,2\pi/3,4\pi/3\}.$
We begin by showing that
\begin{equation}\label{eqn:z3quad} 
	\expec{P\in Z_3}{\det(A+(P\oplus I)B(P\oplus I)^T)}=
	 \expec{P\in Z_3}{\det(A+(PR_\theta\oplus I)B(PR_\theta\oplus I)^T)},\end{equation}
for every $\theta$, where $I$ is the $(d-2$)-dimensional
	 identity.
Since the elements of $Z_3$ are themselves rotations, we can rewrite thrice the right hand
side of \eqref{eqn:z3quad} as
\begin{align*} 
	&\sum_{\tau\in T}\det(A+(R_\tau R_\theta \oplus I)B(R_\tau R_\theta
	\oplus I)^T) 
	\\&= \sum_{\tau\in T}\det(A+(R_{\tau+\theta}\oplus
	I)B(R_{\tau+\theta}\oplus I)^T)
	\\&= \sum_{\tau\in T}\sum_{k=-2}^2
	c_ke^{ik(\tau+\theta)}\quad\textrm{for some coefficients $c_k$, by Lemma \ref{lem:gt3}}
	\\&= \sum_{k=-2}^2 c_ke^{ik\theta}\left(e^{ik0}+e^{ik 2\pi/3}+e^{ik 4\pi/3}\right)
	\\&= 3 c_0 \quad\textrm{since the terms with $|k|=1,2$ vanish}.
\end{align*}
As this quantity is independent of $\theta$, we can assume $\theta=0$,
which gives the left hand side of \eqref{eqn:z3quad}.

To finish the proof, we observe that
\begin{align*}
\expec{P\in A_2}{\det(A+(P\oplus I_{d-2})B(P\oplus I_{d-2})^T)}
&=\expec{D \in F} \expec{P\in Z_3}{\det(A+(PD\oplus I)B(PD\oplus I)^T)}
\\&=\expec{D \in F} \expec{P\in Z_3}{\det(A+(P\oplus I)(D\oplus I)B(D\oplus
I)^T(P\oplus I)^T)},
\end{align*}
where $F$ consists of the identity and the reflection across the horizontal axis:
$$F:=\left\{\begin{bmatrix} 1 & 0\\ 0 & 1\end{bmatrix}, \begin{bmatrix} 1 & 0 \\ 0 &
-1\end{bmatrix}\right\},$$
and $D$ is chosen uniformly from $F$.

Thus, the left hand side of \eqref{eqn:a2invar} is invariant under conjugation of $B$ with the matrices
  $D\oplus I, D\in F$.
Since every $Q_0\in O(2)$ can be written a $R_\theta D$ for some $D\in F$, and we
have already established invariance under $R_\theta\oplus I$ in
  \eqref{eqn:z3quad}, the lemma is proved.

\end{proof}
}{}
\iffull{}{The key structural property of determinants that we exploit is the
  following.}
\begin{lemma}[Determinants are Low Degree in Rank 2 Rotations]\label{lem:gt3}
Let $A, B$ be $d \times d$ matrices and let
\[
\mydet{ A  + \rstt B\rstt^T} = \sum_k c_k e^{i k\theta}.
\]
Then $c_k = 0$ for $|k| \geq 3$.
\end{lemma}
\iffull{\begin{proof}
Recall that all $2\times 2$ rotations may be diagonalized as
$$ R_{\theta}=\begin{bmatrix} \cos\theta &\sin\theta \\ -\sin\theta &
\cos\theta\end{bmatrix}=\ U \begin{bmatrix}e^{i\theta} & 0 \\ 0 & e^{-i\theta}\end{bmatrix}U^\dagger,$$
where 
$$ U=\frac{1}{\sqrt{2}}\begin{bmatrix} 1 & 1 \\ i & -i\end{bmatrix}$$
is independent of $\theta$. This implies that $\rstt =VDV^\dagger$ for diagonal $D$ containing $e^{i\theta}$ and $e^{-i\theta}$
in the upper right $2\times 2$ block and ones elsewhere, with $V$ independent of $\theta$. Thus, we see that
\begin{align*}
\mydet{A  + \rstt B\rstt^T} &= \mydet{ A\rstt + \rstt B} 
	\\&= \mydet{ AVDV^\dagger +VDV^\dagger B}
	\\&= \mydet{ V^\dagger AVD +DV^\dagger B V}
\end{align*}
Notice that the matrix $M=V^\dagger AVD+DV^\dagger BV$ depends linearly on $e^{i\theta},
e^{-i\theta}$, and that the $e^{i\theta}$ (resp. $e^{-i\theta}$) terms appear
only in the first (resp. second) row and column of $M$, respectively. Since each monomial in the expansion of the 
determinant contains at most one entry from each row and each column and
$e^{i\theta}\cdot e^{-i\theta}=1$, this
implies that no terms of degree higher than two in $e^{i\theta}$ or
$e^{-i\theta}$ appear.
\end{proof}
}{}
\iffull{\begin{corollary}[Invariance for $A_{i,j,k}$]\label{cor:aijkinvar} For every $i$, $j$ and $k$,
	$$ \expec{P\in A_{i,j,k}} \det(A+PBP^T) = \expec{Q\in O_{i,j,k}}
		\det(A+QBQ^T).$$
\end{corollary}}
{
\begin{corollary}[Invariance for $A_{i,j,k}$]\label{cor:aijkinvar} For all $i$, $j$ and $k$,
	$$ \expec{P\in A_{i,j,k}} \det(A+PBP^T) = \expec{Q\in O_{i,j,k}}
	\det(A+QBQ^T)$$
\end{corollary}}
\iffull{\begin{proof}
Let $V$ be the orthogonal transformation that maps the affine subspace spanned
    by the vertices $\hat{e}_i, \hat{e}_j, \hat{e}_k$ to the first two
    coordinates of $\R^2$, with any one vertex mapped to a multiple of
    $e_1$.
Conjugation by $V$ maps $A_{i,j,k}$ to $A_2\oplus I_{d-2}$ and
    $O_{i,j,k}$ to $O(2)\oplus I_{d-2}$, abusing notation slightly in the
    natural way. 
Since the determinant is invariant under change of basis, Lemma \ref{lem:a2invar} tells us that
\begin{align*}
	\expec{P\in A_{i,j,k}}\det(A+PBP^T) &= \expec{P_2\in A_2}\det(VAV^T+(P_2\oplus I)VBV^T(P_2\oplus I)^T) 
	\\&= \expec{Q_2\in O(2)}\det(VAV^T + (Q_2\oplus I)VBV^T(Q_2\oplus I)^T)
	\\&= \expec{Q\in O_{i,j,k}}\det(A+QBQ^T),
\end{align*}
as desired.
\end{proof}
}{}
\begin{lemma}[$O_{i,j,k}$ generate $O(d)$]\label{lem:generate}
Given a regular simplex in $\R^d$, the union over $i$, $j$, and $k$ of $O_{i,j,k}$ generates $O (d)$.
In particular, every matrix in $O (d)$ may be written as a product of a finite number of these matrices.
\end{lemma}
\iffull{
\begin{proof}
Let $\Gamma_h $ be the subgroup of $O(d)$ generated by $\bigcup_{i,j,k}
O_{i,j,k\le h}$.
Let $\hat{e}_{0}, \dots , \hat{e}_{d}$ be the vertices of the regular simplex.
For $1 \leq h \leq d$, let $E_{h}$ be the linear subspace parallel to the affine subspace through
  $\hat{e}_{0}, \dots , \hat{e}_{h}$.
We will prove by induction on $h$ that $\Gamma_h$ contains the action of the orthogonal group
  on $E_{h}$.
The base case is $h = 2$, for which $O_{0,1,2}$ is precisely the action of the orthogonal
  group on $E_{2}$.

Assuming that we have proved this result for $h-1$, we now prove it for $h$.
To this end, let $u_{h} = \hat{e}_{h}$, and let $u_{1}, \dots , u_{h-1}$ be
  arbitrary orthonormal vectors in $E_{h}$ that are orthogonal to $u_{h}$.
We will prove that for every orthonormal basis $w_{1}, \dots , w_{h}$ of $E_{h}$,
  there is a $Q \in \Gamma_{h}$ such that $Q w_{i} = u_{i}$ for $1 \leq i \leq h$.

We first consider the case in which $w_{h}^{T} \hat{e}_{h} \geq 0$.
Let $F_h$ denote the 2-dimensional affine subspace spanned by
$\{\hat{e}_{h}, \hat{e}_{h-1}, \hat{e}_{h-2}\}$, and 
observe that there must be a unit vector $p \in E_{h}\cap F_h$ 
  with $p^{T} \hat{e}_{h} = w_{h}^{T} \hat{e}_{h}$.
This follows because the intersection of $F_h$
with the unit sphere in $E_h$ is a circle containing
$\{\hat{e}_h,\hat{e}_{h-1},\hat{e}_{h-2}\}$, $p\mapsto p^T\hat{e}_h$ is a continuous
  function, and we have $\hat{e}_h^T\hat{e}_h=1$ and $\hat{e}_{h-1}^{T} \hat{e}_{h} = \hat{e}_{h-2}^{T} \hat{e}_{h} < 0$.
As $\hat{e}_{h}$ is orthogonal to $E_{h-1}$ and $\hat{e}_{h}$ is invariant under $\Gamma_{h-1}$,
  the induction hypothesis implies that there must be a $T \in \Gamma_{h-1}$ so that $T w_{h} = p$.
Moreover, there is an element $T_2$ of $O_{h-2,h-1,h}$ that maps $p$ to $\hat{e}_{h}$.
So, their composition $W=T_2T$ sends $w_{h}$ to $\hat{e}_{h}$.
Since $W$ is orthogonal, it must send $w_{1}, \dots , w_{h-1}$ to $E_{h-1}$, and so by induction may be composed with a map
in $\Gamma_{h-1}$ that sends $Ww_1,\ldots,Ww_{h-1}$ to $u_{1}, \dots , u_{h-1}$ without moving $\hat{e}_{h}$.
The resulting map is the desired $Q$.

In the case that $w_{h}^{T} \hat{e}_{h} < 0$, we begin by applying a map in $\Gamma_{h}$ that sends
  $w_{h}$ to a vector that is orthogonal to $\hat{e}_{h}$ so that we can then apply the previous argument.
For example, we can do this by defining $p$ to be one of the two unit vectors in
$F_h$ with $p^{T} \hat{e}_{h} = - w_{h}^{T} \hat{e}_{h}$.
We then apply a map in $\Gamma_{h-1}$ that sends $w_{h}$ to $-p$, and then a map in
  $O_{h-2,h-1,h}$ that maps $p$, and thus also $-p$, to a vector orthogonal to $\hat{e}_{h}$.
\end{proof}}{}
\begin{theorem}[Invariance for $A_d$]\label{thm:odinvar}
Let $A$ and $B$ be $d\times d$ matrices, and let
\[
  f_{A,B} (Q) = \mydet{A + Q B Q^{T}}.
\]
Then, for all $Q_{0} \in O (d)$,
\[
\expec{P \in A_{d}}{f_{A,B} (P)} = 
\expec{P \in A_{d}}{f_{A,B} (P Q_{0})}.
\]
\end{theorem}

\iffull{\begin{proof}
We will use the fact that
\[
	\expec{P \in A_{n}}{f_{A,B} (P )} = 
	\expec{P \in A_{d}}{ \expec{P_{2} \in A_{i,j,k}}{f_{A,B} (P P_{2})}}
	= \expec{P\in A_d}\expec{P_2\in A_{i,j,k}}{f_{P^TAP,B}(P_2)}.
\]
Applying Corollary~\ref{cor:aijkinvar} reveals that for every $Q_{2} \in O_{i,j,k}$,
\begin{align*}
	\expec{P\in A_d}\expec{P_2\in A_{i,j,k}}{f_{P^TAP,B}(P_2)}.
	&= \expec{P\in A_d}\expec{P_2\in A_{i,j,k}}{f_{P^TAP,B}(P_2Q_2)}.
	\\&= \expec{P\in A_d}\expec{P_2\in A_{i,j,k}}{f_{A,B}(PP_2Q_2)}.
	\\&= \expec{P\in A_d}{f_{A,B}(PQ_2)}.
\end{align*}
Thus, we conclude that 
$$ \expec{P\in A_d}f_{A,B}(P) = \expec{P\in A_d}f_{A,B}(PQ_2)$$
  for every $Q_2\in O_{i,j,k}$, for every $i,j,k$.

Let $Q_{0} \in O (d)$.
By Lemma \ref{lem:generate}, there is a sequence of matrices
  $Q_{1}, \dotsc , Q_{m}$, each of which is in $O_{i,j,k}$ for some $i$, $j$ and $k$, so that
\[
  Q_{0} = Q_{1} Q_{2} \dotsb Q_{m}.
\]
By applying the previous equality $m$ times, we obtain
\[
	\expec{P \in A_{d}} {f (P Q_{0})}
=
  \expec{P \in A_{d}} {f (P Q_{1} \dotsb Q_{m})}
= 
  \expec{P \in A_{d}} {f (P)}.
\]
\end{proof}

\begin{proof}[Proof of Theorem~\ref{thm:mainquad}]
Follows from Theorem~\ref{thm:odinvar} and Lemma~\ref{lem:invarquad}.
\end{proof}
\begin{proof}[Proof of Theorem \ref{thm:symquadconv}] Follows from Theorem
	\ref{thm:mainquad}, \eqref{eqn:blockdiag}, and
\eqref{eqn:dividehat}.\end{proof}
We conclude the section by recording the obvious extension of Theorem
  \ref{thm:symquadconv} to sums of $m$ matrices.
\begin{corollary}\label{cor:msum} Let $A_1,\ldots, A_m$ be symmetric $d \times d$ matrices  with $A_i\onevec =
	a_i \onevec$ and $\charp{A_i}{x}=(x-a_i)p_i(x)$. Then,
\begin{equation}\label{eqn:msum} 
	\expec{P_1,\ldots, P_m} \charp{\sum_{i=1}^m P_i A_i P_i^T}{x} =
	\left(x-\sum_{i=1}^m a_i\right)p_1(x)\sqsum \ldots\sqsum p_m(x),
\end{equation}
where $P_{1}, \dots , P_{m}$ are independent uniformly random permutation matrices.
\end{corollary}
\begin{proof} Apply a change of basis so that each $A_i=\hat{A_i}\oplus a_i$,
	divide out the $(x-\sum_{i=1}^m a_i)$ term as in \eqref{eqn:dividehat},
	and apply Theorem \ref{thm:mainquad} inductively $(m-1)$ times,
	replacing each $\hat{P_i}$ with a random orthogonal $Q_i$ (this requires conditioning on
	the other $\hat{P_j}$ and $Q_j$, but by independence the distribution of each $\hat{P_i}$ is
	still uniform on $A_d$). Finally, appeal to the identity
	\eqref{eqn:msumid} to write this as an $m-$wise additive convolution.
\end{proof}
}{}

\subsection{Quadrature for Bipartite Matrices}\label{sec:quadbip}

\iffull{}{We include a proof of the analogous result for bipartite matrices in the full version of the paper.
}
\begin{theorem}\label{thm:bipquad}
	Suppose $A$ and $B$ are (not necessarily symmetric)
	$d\times d$ matrices such that $A\onevec = A^T\onevec= a\onevec$ and
	$B\onevec =B^T\onevec =b\onevec$. Let $\charp{AA^T}{x}=(x-a^2)p(x)$ and
	$\charp{BB^T}{x}=(x-b^2)q(x)$. Then,
\begin{align}\label{eqn:bipartitemain} 
		\expec{P,S} \charp{\dil{A}+(P\oplus S)\dil{B}(P\oplus
	S)^T}{{x}} 
  & = \subsquare \left((x-(a+b)^2)p(x)\recsum q(x)\right)
\\
  & = (x^{2}-(a+b)^2) \subsquare \left(p(x)\recsum q(x)\right),
\end{align}
	where $P$ and $S$ are independent uniform random permutation matrices.
\end{theorem}
\iffull{As in the nonbipartite case, we begin by applying a change of basis $V$ that
  isolates the common all ones eigenvector and block diagonalizes our matrices
  as:
  \begin{equation}\label{eqn:bipblockdiag} VAV^T = \hat{A}\oplus a, VBV^T = \hat{B}\oplus b, VPV^T = \hat{P}\oplus 1,
  VSV^T = \hat{S}\oplus 1.\end{equation}
Conjugating the left hand side of \eqref{eqn:bipartitemain} by $(V\oplus V)$, we
see that it is the same as
\begin{align}\label{eqn:bipdividehat}
	\nonumber &\expec{P,S} \charp{\dil{(\hat{A}\oplus a)}+
	((\hat{P}\oplus 1)\oplus(\hat{S}\oplus 1))\dil{(\hat{B}\oplus
	b)}((\hat{P}\oplus 1)\oplus (\hat{S}\oplus 1))^T}{{x}}\\
	&=\nonumber\expec{P,S} \charp{\dil{(\hat{A}+\hat{P}\hat{B}\hat{S}^T\oplus (a+b))}}{{x}} 
	\\\nonumber &=\expec{P,S} \subsquare \charp{(\hat{A}+\hat{P}\hat{B}\hat{S}^T\oplus (a+b)) (\hat{A}+\hat{P}\hat{B}\hat{S}^T\oplus (a+b))^T }{x} 
	\\\nonumber &=(x^{2}-(a+b)^2)\expec{P,S} \subsquare \charp{(\hat{A}+\hat{P}\hat{B}\hat{S}^T) (\hat{A}+\hat{P}\hat{B}\hat{S}^T)^T }{x} 
	\\&=(x^{2}-(a+b)^2)\expec{P,S}
	\charp{\dil{\hat{A}}+(\hat{P}\oplus\hat{S})\dil{\hat{B}}(\hat{P}\oplus\hat{S})^T}{{x}}.
\end{align}
As in the previous section, the matrices $\hat{P}$ and $\hat{S}$ are random elements of
  the group $A_{d-1}$.
Observe that
$$ p(x)=\charp{\hat{A}\hat{A}^T}{x},\quad\text{and}\quad
q(x)=\charp{\hat{B}\hat{B}^T}{x}.$$
Recalling from \eqref{eqn:asymdil} that 
\begin{align*}
	\subsquare\left(p(x)\recsum q(x)\right) 
	&= \E_{Q,R\in O(d-1)}\charp{\dil{A}+(Q\oplus R)\dil{B}(Q\oplus R)^T}{x}
\end{align*}
and removing all the $\hat{\cdot}$s as before to ease notation, 
  we see that the conclusion \eqref{eqn:bipartitemain} of Theorem \ref{thm:bipquad} is
  implied by the following more general quadrature statement.
\begin{theorem}\label{thm:dilquad} 
For all symmetric $2d\times 2d$ matrices $C$ and $D$:
	\begin{equation}	
		\expec{P,S\in A_d}\charp{C+(P\oplus S)D(P\oplus S)^T}{{x}} = 
		\expec{Q,R\in O(d)}\charp{C+(Q\oplus R)D(Q\oplus R)^T}{{x}}.
	\end{equation}
\end{theorem}
This theorem is an immediate consequence of two applications of the following
  corollary of Theorem \ref{thm:mainquad} from the previous section.
  \begin{corollary}\label{cor:oplusquad} If $C$ and $D$ are symmetric $2d\times 2d$ matrices, 
	$$ \expec{P\in A_d} \det(C+(P\oplus I)D(P\oplus I)^T) 
	= \expec{Q\in O(d)} \det(C+(Q\oplus I)D(Q\oplus I)^T).$$
\end{corollary}
\begin{proof} The proof is identical to the proof of Theorem \ref{thm:mainquad},
	except we replace $P\in A_d$ with $P\oplus I$ and $Q\in O(d)$ with
	$Q\oplus I$ at each step.

	Specifically, let
	$$ f_{C,D}(Q):=\det(C+(Q\oplus I)D(Q\oplus I)^T).$$
	Applying Corollary \ref{cor:aijkinvar} as before reveals that for every
	$i,j,k$ and every $Q_2\in O_{i,j,k}$,
	$$ \expec{P\in A_d} f_{C,D}(P) = \expec{P\in A_d}\expec{P_2\in
	A_{i,j,k}}f_{C,D}(PP_2) = \expec{P\in A_d}\expec{P_2\in
	A_{i,j,k}}f_{C,D}(PP_2Q_2) = \expec{P\in A_d}f_{C,D}(PQ_2).$$
Since an arbitrary $Q_0\in O(d)$ is a
	finite product of such $Q_2$ by Lemma \ref{lem:generate}, this means
	that
	$$ \expec{P\in A_d}f_{C,D}(PQ_0)=\expec{P\in A_d}f_{C,D}(P)$$
	for all $Q_0\in O(d)$. Lemma \ref{lem:invarquad} completes the proof.
\end{proof}

\begin{proof}[Proof of Theorem \ref{thm:dilquad}]
	Since $P$ and $S$ are independent, we have
	\begin{align*}
		&\expec{P,S\in A_d}\charp{C+(P\oplus S)D(P\oplus S)^T}{{x}}
		\\&= \expec{S\in A_d}\expec{P\in A_d}\det({x}I+C+(P\oplus
		I)(I\oplus S)D(I\oplus S)^T(P\oplus I)^T)
		\\&= \expec{S\in A_d}\expec{Q\in O(d)} \det({x}I+C+(Q\oplus
		I) (I\oplus S)D(I\oplus S)^T (Q\oplus I)^T)\quad\textrm{by
		Corollary \ref{cor:oplusquad}}
		\\&= \expec{Q\in O(d)} \expec{S\in A_d}\det({x}I+(Q\oplus
		I)^TC(Q\oplus I)+(I\oplus S)D(I\oplus S)^T)
		\\&= \expec{Q\in O(d)} \expec{R\in O(d)}\det({x}I+(Q\oplus
		I)^TC(Q\oplus I)+(I\oplus R)D(I\oplus R)^T)\quad\textrm{by
		Corollary \ref{cor:oplusquad}}
		\\&= \expec{Q,R\in O(d)} \det({x}I+C+(Q\oplus R)D(Q\oplus
		R)^T),
	\end{align*}
as desired.
\end{proof}

\begin{proof}[Proof of Theorem \ref{thm:bipquad}] Follows from Theorem
	\ref{thm:dilquad}, \eqref{eqn:bipblockdiag}, and
	\eqref{eqn:bipdividehat}.
\end{proof}

As before, Theorem 4.10 extends effortlessly to the case of many matrices.
\begin{corollary}\label{cor:bipquad}
If $A_1,\ldots,A_m$ are matrices with $A_i\onevec =A_i^T\onevec = a_i$
	and $\charp{A_iA_i^T}{x}=(x-a_i^2)p_i(x)$, then
	\begin{align*}
		&\expec{P_1,\ldots, P_m, S_1,\ldots, S_m} \charp{\sum_{i=1}^m (P_i\oplus S_i)\dil{A_i}(P_i\oplus S_i)^T}{{x}} 
		\\&= \left(x^{2}-\left(\sum_i a_i\right)^2\right) \subsquare\left[
	p_1(x)\recsum\ldots\recsum p_m(x)\right],
	\end{align*}
where the $P_i$ and $S_i$ are independent uniformly random permutations.
\end{corollary}
We omit the proof, which is identical to the proof of Corollary \ref{cor:msum}.
}{}

\section{Ramanujan Graphs}\label{sec:bounds}
In this section, we combine the Cauchy transform, interlacing, and quadrature
  results of the previous sections to establish our main Theorems
  \ref{thm:bipartite} and \ref{thm:nonbipartite}

\begin{proof}[Proof of Theorem \ref{thm:nonbipartite}]
Let $M$ be the adjacency matrix of a fixed perfect matching on $d$ vertices, with $d$
  even.
Since the uniform distribution on permutations is realizable by swaps (Lemma
  \ref{lem:realizable}), Theorem \ref{thm:perminterlace} tells us that with nonzero probability: 
 	$$\lambda_2\left(\sum_{i=1}^d P_iMP_i^T\right)\le
	\lambda_2\left(\E\charp{\sum_{i=1}^m P_i M P_i^T}{x}\right).$$

	\iffull{Corollary \ref{cor:msum} reveals}{Repeated applications of
Theorem \ref{thm:symquadconv} reveal} that 
  the polynomial in the right-hand expression
   may be written as an $m$-wise symmetric additive
	  convolution\footnote{We remark that this formula works for arbitrary regular
	 	 adjacency matrices $A_i$, as mentioned in the abstract.}
 $$ E(x):=\expec{P_1,\ldots,P_m}\charp{\sum_{i=1}^m P_i A_i P_i^T}{x}=
(x-m)p(x)\sqsum_{d-1} \ldots\sqsum_{d-1} p(x)\qquad\text{($m$ times)},$$
where
$$p(x)=\frac{\charp{x}{M}}{x-1}=(x-1)^{d/2-1}(x+1)^{d/2},$$
is the characteristic polynomial of a single matching with the trivial root at
  $1$ removed.
Our goal is therefore to bound the largest root of $p(x)\sqsum \ldots\sqsum
p(x)$,
  which is the second largest root of $E(x)$.

We will do this using the inverse Cauchy transform described in Section
  \ref{sec:cauchyintro}.
The Cauchy transform of $p(x)$ is given by
$$ \cauchy{p}{x}=\frac{d/2-1}{d-1}\frac{1}{x-1}+\frac{d/2}{d-1}\frac{1}{x+1}.$$
Notice that for every $x>1$, putting the trivial root at $1$ back only increases
the Cauchy transform:
\begin{equation}
\label{eqn:addone} \cauchy{p}{x} <
	\frac{d/2}{d}\frac{1}{x-1}+\frac{d/2}{d}\frac{1}{x+1}=\frac{x}{x^2-1} =
\cauchy{\chi(M)}{x}.\end{equation}
Since both functions are decreasing for $x>1$, this implies that the inverse
Cauchy transform of $p$ is upper bounded by that of $\chi(M)$:
$$\invcauchy{p}{w}<\invcauchy{\chi(M)}{w},$$
for every $w>0$.

Applying the convolution inequality in Theorem \ref{thm:sqsumTrans} $(m-1)$
times yields the following upper bound on the inverse Cauchy transform of the $m-$wise convolution of interest.
\begin{equation}\label{eqn:finalcalc}\invcauchy{p\sqsum\ldots\sqsum p}{w}\le m\cdot \invcauchy{p}{w}-\frac{m-1}{w}
< m\cdot \invcauchy{\chi(M)}{w}-\frac{m-1}{w}.\end{equation}
Recalling from \eqref{eqn:addone} that
$$ \invcauchy{\chi(M)}{w}=x\iff w=\frac{x}{x^2-1},$$
the right hand side of \eqref{eqn:finalcalc} may be written as
$$ mx-\frac{m-1}{w} =
mx-\frac{(m-1)(x^2-1)}{x}=\frac{x^2+(m-1)}{x},$$
which is easily seen to be minimized at $x=\sqrt{m-1}$ with value $2\sqrt{m-1}$.
Thus, the second largest root of $E (x)$ is at most $2 \sqrt{m-1}$.
\end{proof}

\iffull{\begin{proof}[Proof of Theorem \ref{thm:bipartite}]
Let 
$$M=\dil{I}$$
be the adjacency matrix of a perfect matching on $2d$ vertices, across the
  natural bipartition.
Then, for independent uniformly random $d\times d$ permutation matrices $P_1,\ldots,P_m,
S_1,\ldots,S_m$, the random matrix
$$ A = \sum_{i=1}^m (P_i\oplus S_i)M(P_i\oplus S_i)^T = \sum_{i=1}^m
\dil{(P_iS_i^T)}$$
  is the adjacency matrix of a union of $m$ random matchings across the same
  bipartition.
Since the distribution of the $(P_i\oplus S_i)$ is realizable by swaps (Lemma
  \ref{lem:realizable}), Theorem \ref{thm:perminterlace} implies that
$$ \lambda_2(A)\le \lambda_2\left(\expec{}\charp{\sum_{i=1}^m (P_i\oplus
S_i)M(P_i\oplus S_i)}{x}\right),$$
with nonzero probability. 
Since $I\onevec = \onevec$, Corollary \ref{cor:bipquad} implies that the
  polynomial on the right hand side is equal to 
  $$ (x^2-m^{2})\subsquare[ p(x)\recsum_{d-1}\ldots\recsum_{d-1} p(x)]\quad\text{($m$ times)},$$
where
$$p(x)=\charp{I_{d-1}I_{d-1}^T}{x} = (x-1)^{d-1}.$$
We upper bound the inverse Cauchy transform of this $m-$wise convolution using Theorem
\ref{thm:recsumTrans}:
$$ \invcauchy{\subsquare (p\recsum\ldots\recsum p)}{w}\le
m\cdot\invcauchy{\subsquare p}{w}-\frac{m-1}{w}= m\cdot\invcauchy{(x^2-1)^{d-1}}{w}-\frac{m-1}{w}.$$
Since
$$\cauchy{(x^2-1)^{d-1}}{w} = \frac{x}{x^2-1},$$
this is now identical to the calculation \eqref{eqn:finalcalc}, so we obtain
  again the bound $2\sqrt{m-1}$.
Thus, we conclude that $\lambda_2(A)\le 2\sqrt{m-1}$ with nonzero probability.
Since $A$ is bipartite, its spectrum is symmetric about zero, so we must also
have $\lambda_{d-1}(A)\ge -2\sqrt{m-1}$, whence $A$ is Ramanujan.
\end{proof}
}{
The proof of Theorem \ref{thm:bipartite} is analogous, and appears in the full version.}

\bibliographystyle{alpha}
\bibliography{free}

\end{document}